\documentclass[leqno]{amsart}
\usepackage{srctex,color}

\renewcommand{\Re}{\operatorname{Re}}
\newcommand{\Tr}{\mathrm{Tr}}

\newcommand{\dd}{\mathrm{d}}

\newcommand{\der}[3]{\displaystyle{\frac{\partial^{#3} #1}{\partial #2^{#3}}}}

\newcommand{\R}{\mathbb R}
\newcommand{\E}{\mathbb E}
\newcommand{\C}{\mathbb C}

\newcommand{\ep}{\varepsilon}

\newcommand{\TT}{{\mathcal T}}
\newcommand{\DD}{{\mathcal D}}
\newcommand{\cD}{{ \mathcal D}}
\newcommand{\inner}[3][]{(#2,#3 )_{#1}}

\newtheorem{theorem}{Theorem}[section]
\newtheorem{lemma}[theorem]{Lemma}
\newtheorem{proposition}[theorem]{Proposition}
\newtheorem{hypothesis}{Assumption}
\newtheorem{corollary}[theorem]{Corollary}

\theoremstyle{definition}

\theoremstyle{remark}
\newtheorem{remark}[theorem]{Remark}

\numberwithin{equation}{section}

\begin{document}

\title[Strong order for linear stochastic Volterra equations]{Strong order of convergence of a fully discrete approximation of a linear stochastic Volterra type evolution equation}

\author{Mih\'aly Kov\'acs}
\address{Department of Mathematics and Statistics, University of Otago, PO Box 56, Dunedin, 9054, New Zealand.}
\email{mkovacs@maths.otago.ac.nz}

\author{Jacques Printems}
\address{Laboratoire d'Analyse et de Math\'ematiques Appliqu\'ees, CNRS UMR 8050, 61, avenue du G\'en\'eral de Gaulle, Universit\'e Paris--Est, 94010 Cr\'eteil, France.}
%\curraddr{Department of Mathematics and Statistics,
%Case Western Reserve University, Cleveland, Ohio 43403}
\email{printems@u-pec.fr}
\date{}


%    General info
\subjclass[2000]{34A08, 45D05, 60H15, 60H35, 65M12, 65M60}
\keywords{Stochastic Volterra equation, fractional differential equation, finite elements method, convolution quadrature, Euler scheme, strong order}

\begin{abstract}
In this paper we investigate a discrete approximation in time and in space of a Hilbert space valued stochastic process $\{u(t)\}_{t\in [0,T]}$ satisfying a stochastic linear evolution equation with a positive-type memory term  driven by an
additive Gaussian noise. The equation can be written in an abstract form as
$$
\dd u +  \left ( \int_0^t b(t-s) Au(s) \, \dd s \right )\, \dd t  = \dd W^{_Q},~t\in (0,T]; \quad u(0)=u_0 \in H,
$$
\noindent where $W^{_Q}$ is a $Q$-Wiener process on $H=L^2({\mathcal D})$ and where the main example of $b$ we consider is given by
$$
b(t) = t^{\beta-1}/\Gamma(\beta), \quad 0 < \beta <1.
$$
We let $A$ be an unbounded linear self-adjoint positive operator on $H$ and we further assume that there exist $\alpha >0$ such that $A^{-\alpha}$ has finite trace and that $Q$ is bounded from $H$ into $D(A^\kappa)$ for some real $\kappa$ with $\alpha-\frac{1}{\beta+1}<\kappa \leq \alpha$.

The discretization is achieved via an implicit Euler scheme and a Laplace transform convolution quadrature in time (parameter $\Delta t =T/n$), and a standard continuous finite element method in space (parameter $h$). Let
$u_{n,h}$ be the discrete solution at $T=n\Delta t$. We show that
$$
\left ( \E \| u_{n,h} - u(T)\|^2 \right )^{1/2}={\mathcal O}(h^{\nu} + \Delta t^\gamma),
$$
%\noindent for any  $\gamma< (1 - \alpha + \kappa - \beta)/2\leq 1- \beta/2$ and $\nu \leq  \frac{1}{\beta+1}-\alpha+\kappa\le 1$.
\noindent for any  $\gamma< (1 - (\beta+1)(\alpha - \kappa))/2 $ and $\nu \leq  \frac{1}{\beta+1}-\alpha+\kappa$.
\end{abstract}
\maketitle

\section{Introduction}

%In this paper, we are interested in the convergence of numerical approximations of the following type of stochastic linear evolution equation: given $T>0$, $d\geq 1$
Let
$\mathcal D$ be a bounded domain in $\R^d$, $d\in \mathbb{N}$, and let $u$ be a real-valued stochastic process solution of the equation formally written as
\begin{equation}\label{eq:0}
\der{u(x,t)}{t}{} - \int_0^t b(t-s) \Delta u(x,s) \, \dd s = \dot{\xi}(x,t), \quad (x,t) \in {\mathcal D}\times(0,T],
\end{equation}
\noindent together with the initial condition $u(x,0) = u_0(x)$, $x\in \mathcal D$, and boundary condition $u|_{\partial \mathcal{D}}=0$. Here, $\dot{\xi}$ is a zero mean real valued Gaussian noise and the time kernel $b$ is assumed to be real-valued and of positive type; i.e.,
that for any $T>0$, the kernel $b$ belongs to $L^1(0,T)$ and for any continuous function $f$ on $[0,T]$ the following inequality holds:
\begin{equation}\label{eq:1}
\int_0^T \int_0^t b(t-s) f(s) f(t) \, \dd s \,  \dd t \geq 0.
\end{equation}
The deterministic version of such problems can be used to model viscoelasticity or heat conduction in materials with memory (see \cite{McLThomee93} for references). When $b$ is smooth,
 these equations exhibit a hyperbolic behaviour, whereas if $b$ has a weak singularity at $t=0$ (for example a Riesz potential), they exhibit certain parabolic features. In particular, when
\begin{equation}\label{eq:riesz}
b(t) = t^{\beta-1}/\Gamma(\beta), \quad 0 < \beta < 1,
\end{equation}
\noindent the homogeneous deterministic equation has a smoothing property which correspond to the inequality
\begin{equation}\label{eq:regular}
\| u^{(m)}(t)\|_{H^{2r}(\R)} \leq C \, t^{-(\beta+1)r -m} \|u_0\|_{L^2({\mathcal D})},
\end{equation}
 \noindent where $|r|\leq 1$ if $m\ge 1$ and where $0\leq r \leq 1$ if $m=0$, but with no further smoothing in the spacial variables (see e.g. \cite[Theorem 5.5]{McLThomee93}). The framework of this paper allows for slightly more general kernels but with
 similar smoothing effects and, in particular, they are of positive type. Hence, together with the positivity of the operator $-\Delta$, the deterministic equation will remain parabolic in character.

Next we introduce an abstract framework to describe the noise and equation \eqref{eq:0} more precisely.
Let $Q$ be a self-adjoint, nonnegative linear operator on $H:=L^2({\mathcal D})$ and $W^Q$ be a Wiener process in $H$ with covariance operator $Q$ (or, simply, $Q$-Wiener process).
 We set $A = -\Delta$, $D(A) = H^2({\mathcal D}) \cap H^1_0({\mathcal D})$ and $H= L^2({\mathcal D})$. Then $A$ can be seen as an unbounded linear operator on $H$ with domain $D(A)$. For $b$ given in \eqref{eq:riesz} and under reasonable assumptions on $\partial D$, our main assumption concerning $Q$ in \eqref{eq:0} is that $A^{\kappa} Q$ defines a bounded operator on $L^2({\mathcal D})$ with $d/2 -1/(\beta+1)<\kappa<d/2$.

If we write $u(t) = u(\cdot,t)$, considered as a $H$-valued stochastic process, then \eqref{eq:0} can be rewritten in the abstract It\^o form as
\begin{equation}\label{eq:2}
\dd u(t) + \left ( \int_0^t b(t-s) Au(s) \, \dd s \right ) \, \dd t = \dd W^Q(t), \quad t\in (0,T],
\end{equation}
\noindent with initial condition $u(0) = u_0 \in H$.

While the literature on numerical methods for deterministic infinite dimensional Volterra equations is abundant (see, for example, \cite{Pal,HaNo2011,Lubich_et_al1996,McLThomee93,xu2008}, which is a very incomplete list), the numerical analysis of stochastic Volterra equations is completely missing. We are only aware of \cite{KaRo} where an algorithm is described and numerical experiments are performed with no error analysis given.
We will consider a numerical approximation of \eqref{eq:2} by means of an Euler scheme and a Laplace transform convolution quadrature in time together
with a finite element method in space. Let $n\geq 1$ an integer, $\Delta t = T/n$ and $t_k = k \, \Delta t$, $k=0, \dots, n$. Let also $\{ V_h\}_{h>0}$ be a family of
finite dimensional subspaces of $D(A^{1/2})=H^1_0(\mathcal{D})$. For each $1\leq k \leq n$, we seek for an approximation of $u(t_k)$ in $V_h$ by $u_{k,h}$ defined by the following induction:
\begin{equation}\label{eq:3}
(u_{k,h} - u_{k-1,h},v_h) + \Delta t \sum_{j=1}^{k} \omega_{k-j} ( A u_{j,h},v_h) = \sqrt{\Delta t} (Q^{1/2} \chi_{k},v_h),\quad
k\geq 1,
\end{equation}
\noindent for any $v_h\in V_h$, where $\sqrt{\Delta t}\,  \chi_k $ is the noise increment and where $(\cdot,\cdot)$ is the inner product of $H$. The approximation of the convolution term in \eqref{eq:2} is achieved via a quadrature rule such that for any continuous function $f$ on $[0,T]$,
$$
\sum_{j=1}^k \omega_{k-j} f(t_j)  \sim \int_0^{t_k} b(t_k-s) f(s) \, \dd s=(b\star f)(t_k).
$$
 Then, the approximation of $b\star f$ on the time grid $t_k$, $k=0,\dots,n$, is obtained from a discrete convolution with the values of $f$ on the same grid. Before going into details, let us point out that not any quadrature rule can be chosen. In particular, it will be important for the chosen quadrature to satisfy a discrete analogue of \eqref{eq:1}.

 In order to understand the specific quadrature rule used in this paper, we will take the example of the Riesz kernel \eqref{eq:riesz}. Let us note that in this case the Laplace transform
of $b$ is $z^{-\beta}$ and the term $b\star \Delta u$ in \eqref{eq:0} can be seen as the fractional integral $(\partial /\partial t)^{-\beta} (\Delta u)$. Then, the
idea is to
use the same Euler approximation of $\partial/\partial t$
in both terms on the left hand side of (\ref{eq:0}). Since the discrete Laplace transform of the implicit Euler scheme is $(1-z)/\Delta t$, one chooses the quadrature weights to have discrete Laplace transform  $((1-z)/\Delta t)^{-\beta}$.

Such a convolution quadrature has been introduced in \cite{lubich88,lubich88II}. It was motivated by the fact that the main properties of the solution of the homogeneous
problem, like stability, existence, or regularity, are largely determined by the distribution of the frequencies of the kernel (by means of its Fourier or Laplace transform), especially when the
kernel has weak singularities or when it exhibits different behaviour at different time scales. Since, by construction, the discrete Laplace transform of the quadrature kernel is closely related to the Laplace transform
of the continuous kernel, it is thus possible to carry over frequency domain conditions from the continuous problem to the discretization and thereby obtain stable approximations.
Moreover, this kind of quadrature rule inherits the rate of approximation from the time integrator of $\partial/\partial t$. In the context of stochastic PDEs, we think that it is important to make sure that the deterministic part of the scheme is stable and that the perturbations are due to the noise only.

Although the analysis in the present paper allows for kernels slightly more general than \eqref{eq:riesz}, we follow the same idea: the convolution
quadrature weights $\{\omega_k\}$ in \eqref{eq:3} will be defined by means of the Laplace transform of the kernel $b$. Therefore, we choose
the quadrature coefficients to have generating function $\widehat b((1-z)/\Delta t)$ where $\widehat b$ denotes the Laplace transform of $b$; that is,
\begin{equation}\label{eq:omega}
\sum_{n=0}^{+\infty} \omega_k z^k = \widehat b\left ( \frac{1-z}{\Delta t} \right ), \quad |z|<1.
\end{equation}
We will not focus here on practical algorithms for the computations of the quadrature weights and we refer the reader to, for example, \cite{lubich88II}.

 While precise conditions on the kernel $b$ are postponed to Sections \ref{sec:prel} and \ref{sec:time}, we can already state our main result, Theorem \ref{theo:full}, with the above notations in the case of the specific kernel \eqref{eq:riesz} when $\cD$ is a convex polygonal domain using continuous, piecewise linear finite elements. We shall prove a (strong) error estimate of the form %${\mathcal O}(\Delta t^{\gamma} + h^{\nu})$
$$
\left( \E ( \| u_{n,h} - u(T) \|^2 \right)^{1/2} \leq C (\Delta t^\gamma + h^\nu),
$$
where $\gamma < (1 - (\beta+1)(d/2 - \kappa))/2$ and $\nu < 1/(\beta+1) - d/2 + \kappa$. Let us note that we recover the known order of convergence for the heat equation (see \cite{LLK2010,P2001,yan}) when $\beta\to 0$.

The paper is organized as follows. In Section \ref{sec:prel} we introduce notations, recall some basic preliminary results, and state our main assumptions on $A$, $Q$ and $b$. We note that Assumptions (\ref{eq:traceA})--(\ref{eq:Q}) on $A$ and $Q$ could be replaced by a single, somewhat sharper, assumption as discussed in Remarks \ref{rem:HSassumption}, \ref{rem:spaceaq}, \ref{rem:timeaq} and \ref{rem:fullaq}. It is, however, harder to check in most cases.
In Section \ref{sec:space} we study the space semi-discretization of \eqref{eq:0} and strong error estimates are derived
for smooth initial data under minimal regularity assumptions (Assumption \ref{hyp:b}) on $b$. In Section \ref{sec:time} we prove strong error
estimates for the time
semi-discrete scheme with non-smooth initial data. One of the key results in this direction is Theorem \ref{thm:lc}, where we prove a general $l^p$-stability result on Lubich's convolution quadrature based on the Backward Euler method for deterministic Volterra equations. Interestingly, this stability result implies (Corollary \ref{cor:smooth}) that the time-discrete scheme exhibits the same smoothing effect in time as the solution under Assumption \ref{hyp:b} on $b$. However, in order to obtain optimal convergence rates for the stochastic problem we need to put a further regularity restriction on $b$ in Subsection \ref{sub:dete}, Assumption \ref{hyp:b-anal}, which is in fact common in the deterministic literature for nonsmooth initial data. Indeed, Assumption \ref{hyp:b-anal} implies that the deterministic equation has an analytic resolvent family while Assumption \ref{hyp:b} only implies that the deterministic equation is parabolic. Unlike for equations with no memory term, these two notions are not equivalent (See \cite[Chapter 1, Section 3]{pruss}). As far as we know the derivation of nonsmooth initial data estimates using only parabolicity (Assumption \ref{hyp:b}) remains an open problem.
Finally, in the last section, we gather the results from the preceding sections and consider the fully discrete scheme.

%%%%%%%%%%%%%%%%%%%%%%%%%%%%%%%%%%%%%%
%  II. NOTATIONS and PRELIMINAIRIES
%%%%%%%%%%%%%%%%%%%%%%%%%%%%%%%%%%%%%%
\section{Notations and preliminairies}\label{sec:prel}

Let $(X,\|\cdot\|_X)$ and $(Y,\|\cdot\|_Y)$ be two Banach spaces and let ${\mathcal B}(X,Y)$ denote the space of bounded linear operators from $X$ into $Y$ endowed with the norm
$\|B\|_{{\mathcal B}(X,Y)} = \sup_{x\in X} \| Bx\|_Y/\|x\|_X$. When $X=Y$, we use the shorter notation ${\mathcal B}(X)$ for ${\mathcal B}(X,X)$.
 If $X$ is a Banach space and $I$ is an interval in $\R$ then, $L^p(I,X)$, $1\le p<\infty$, denotes the space of functions $f:I\to X$ which are measurable and $t\to \|f(t)\|^p$ is integrable on $I$, equipped with the usual norm. If $p=\infty$ then $L^\infty(I,X)$, denotes the space of functions $f:I\to X$ which are measurable and $t\to \|f(t)\|$ is essentially bounded on $I$ endowed with the usual norm.

%We shall also denote by ${\mathcal C}(I,X)$ the Banach space of continuous functions from $I$ to $X$ endowed by the norm $\|f\|_{{\mathcal C}(I,X)} = \sup_{t\in %I}\|f(t)\|_X$.

Throughout this paper, $H$ denotes a real separable Hilbert space with inner product $(\cdot,\cdot)$ and associated norm $\|\cdot\|$. We consider the stochastic Volterra equation given in the abstract It\^o form as
\begin{equation}\label{eq:stovolterra}
\dd u + \left ( \int_0^t b(t-s) Au(s) \, \dd s \right )\, \dd t  = \dd W^{_Q},\quad t\in (0,T]; \quad u(0) = u_0 \in H,
\end{equation}
\noindent where  the process $\{u(t)\}_{t \in [0,T]}$ is
a $H$-valued stochastic process, $A$ is a densely defined, nonnegative self-adjoint unbounded operator on $H$ with compact inverse, and $W^{_Q}$ is a
$Q$-Wiener process in $H$ on a given probability space $(\Omega,{\mathcal F},{\mathbb P})$. The weak solution of \eqref{eq:stovolterra} is a mean-square continuous $H$-valued process satisfying
\begin{equation*}
  \left( u(t), \eta \right)
 + \int_0^t \int_0^rb(r-s)\left( u(s), A^*\eta\right)\,\dd s\,\dd r
=\left( u_0,\eta\right)
 + \int_0^t \left(\eta, \,\dd W^{Q}(s) \right),
\end{equation*}
for all $\eta \in D(A^*)$ almost surely for all $t\in [0,T]$.

It is well known that such assumptions on $A$ implies the existence of a sequence of nondecreasing positive real numbers $\{\lambda_k\}_{k\geq 1}$ and an orthonormal basis
$\{e_k\}_{k\geq 1}$ of $H$ such that
\begin{equation}\label{eq:spectral}
Ae_k = \lambda_k e_k, \quad \lim_{k\rightarrow +\infty} \lambda_k = +\infty.
\end{equation}
%Thanks to this spectral decomposition, the cylindrical Wiener process $W(t)$ on $H$ can formally be written as
%$$
%W(t) = \sum_{k\geq 1} \beta_k(t) e_k,
%$$
%\noindent where $\{\beta_k\}_{k\geq 1}$ denotes a family of mutually independent real Brownian motions and then $W^{_Q}(t) = Q^{1/2} W(t)$.

We define classically, by means of the spectral decomposition of $A$, the domains  $D(A^s)$ of fractional powers $s\in \R$ of $A$ and we set
$$
\|v\|_{s} = \| A^{s/2} v\|, \quad v \in D(A^{s/2}).
$$
%Furthermore, for any reals $s_1\leq s \leq s_2$, one has the continuous embeddings $D(A^{s_2/2}) \subset D(A^{s/2}) \subset D(A^{s_1/2})$ and by H\"older's inequality
%\begin{equation}\label{eq:holder}
%\|v\|_s \leq \|v\|^{1-\lambda}_{s_1} \|v\|^\lambda_{s_2}, \quad s = (1-\lambda)s_1 + \lambda s_2,\quad v \in D(A^{s_2/2}).
%\end{equation}

%ADD SOMETHING ABOUT THE COMPLEX VERSION OF THE REAL HILBERT H : DEFINE THE PRODUCT $z h$ WHERE $z\in \C$ AND $h\in H$. CHANGE NOTATION : $H \rightarrow {\mathcal H}$

\begin{remark}
We note that since $A$ is nonnegative self-adjoint, $-A$ generates an analytic contraction semigroup on $H$.
Moreover,  for any $\theta < \pi$,  there exists $M_\theta\geq 1$ such that the following holds:
$$
\| (z I + A)^{-1}\|_{{\mathcal B}(H)} \leq \frac{M_\theta}{|z|}, \quad \mbox{for any }\; z \in \Sigma_\theta,
$$
%Moreover,  for any $\theta < \pi/2$,  the following holds
%$$
%\| (z I + A)^{-1}\|_{{\mathcal B}(H)} \leq \frac{1}{|z|}, \quad \mbox{for any }\; z \in \Sigma_\theta,
%$$
\noindent where $\Sigma_\theta = \{ z \in \C\backslash\{0\}, \; |\mathrm{arg}(z)| < \theta \}$.
%Eventually, by an H\"older interpolation (\ref{eq:holder}) between the last two inequalities, for any $\theta<\pi$ and for any $\delta \in [0,1]$, we have
%\begin{equation}\label{eq:resolvA}
%\| A^\delta (zI + A)^{-1} \|_{{\mathcal B}(H)} \leq M_\theta \, |z|^{-(1-\delta)}, \quad \mbox{for any  } \; z\in \Sigma_\theta.
%\end{equation}
%Ineq. (\ref{eq:resolvA}) will be used in an essential way in this paper.
\end{remark}

Let $\mathcal{L}_1(H)$ denote the set of nuclear operators from $H$ to $H$; that is, $T\in \mathcal{L}_1(H)$ if there are sequences $\{a_j\},\{b_j\}\subset H$ with $\sum_{j=1}^{\infty}\|a_j\| \|b_j\|<\infty$ and such that
\begin{equation*}
%\label{nuc}
  Tx=\sum_{j=1}^{\infty}( x, b_j) a_j,
  \quad x\in H.
\end{equation*}
Sometimes these operators are referred to as {trace class} operators.
For $T\in \mathcal{L}_1(H)$ we define $\mathrm{Tr}(T)$, the trace of $T$, by
$$
\mathrm{Tr}(T) = \sum_{n=1}^{+\infty} (Be_n,e_n),
$$
where $\{e_n\}$ is an orthonormal basis of $H$. This definition turns out to be independent of the choice
basis. Furthermore, if $L\in \mathcal{L}_1(H)$ and $M\in \mathcal{B}(H)$, then $LM,ML\in \mathcal{L}_1(H)$ and
\begin{equation} \label{eq:trace_prop1}
\Tr{}(LM) = \Tr{}(ML).
\end{equation}
\noindent If $L$ is also symmetric and nonnegative, then
\begin{equation} \label{eq:trace_prop2}
\Tr{}(LM) \leq \Tr{}(L) \|M\|_{{\mathcal B}(H)}.
\end{equation}

Hilbert-Schmidt operators play also an important role in this paper.
An operator $L \in {\mathcal B}(H)$ is Hilbert-Schmidt if $L^*L\in \mathcal{L}_1(H)$  or,
equivalently, $LL^*\in \mathcal{L}_1(H)$. We denote by ${\mathcal L}_2(H)$ the space of
such operators. It is a Hilbert space under the norm
\begin{equation}
\label{e2.8a}
\|L\|_{{\mathcal L}_2(H)}= \left( \Tr{} (L^*L)\right)^{1/2}=  \left( \Tr{} (L L^*)\right)^{1/2}.
\end{equation}
%It is also well-known that, given four Hilbert spaces $K_1,\; K_2,\; K_3,\; K_4$,  if $L\in {\mathcal L}_2(K_2,K_3)$,
%$M\in {\mathcal B}(K_1,K_2)$, $N\in {\mathcal B}(K_3,K_4)$ then $NLM\in {\mathcal L}_2(K_1,K_4)$ and
%\begin{equation} \label{eq:HS}
%\| NLM\|_{{\mathcal L}_2(K_1,K_4)}\le \| N\|_{{\mathcal B}(K_3,K_4)}
%\|L\|_{{\mathcal L}_2(K_2,K_3)}\|M\|_{{\mathcal B}(K_1,K_2)}.
%\end{equation}

Our analysis will also use the Laplace transform. Let $f : \R_+  \rightarrow H$ be subexponential; i.e., that for any
$\ep>0$ the function $t\mapsto f(t) e^{-\ep t}$ belongs to $L^1(\R_+,H)$. We define the Laplace transform of
$\widehat f:\C_+ \rightarrow H$ by
$$
\widehat f(z) = \int_0^{+\infty} f(t) e^{-z t}\, dt,\quad \mathrm{Re}\,z > 0,
$$
\noindent where we have used the same notation $H$ for the complexification of $H$.
We denote by $\star$ the Laplace convolution product on $[0,t]$ of two locally integrable subexponential functions $f,g \in L^1_{loc}(\R_+,H)$ defined as
$$
(f\star g)(t) = \int_0^t f(t-s)g(s) \, \dd s.
$$
It is well known that $f\star g\in L^1_{loc}(\R_+,H)$ is subexponential and
$$
\widehat{f\star g}\, (z) = \widehat f(z) \, \widehat g(z),\quad \mathrm{Re}\,z>0.
$$

%SAY SOMETHING ABOUT THE GENERATING FUNCTION AND ITS LINK WITH LAPLACE TRANSFORMS
%Similarly, let $\{\omega_n\}_{n\geq 0}$ be an integrable sequence, we denote by $\widetilde \omega(z)$ its generating function:
%$$
%\widetilde \omega(z) = \sum_{n\geq 0} \omega_n z^n, \quad |z|<1.
%$$
%We will use the following version of Plancherel equality:
%$$
%\int_0^{2\pi}  | \widetilde \omega(e^{ix})|^2 \, dx = c \sum_{n\geq 0} |\omega_n|^2.
%$$

\subsection{Main assumptions}
Next we state the main assumptions concerning the kernel $b$ and the operators $A$ and $Q$, which will be used throughout this paper.

Regarding $b$, first note that property \eqref{eq:1} can be
characterized by means
of the Laplace transform $\widehat b$ of $b$. It is equivalent to say that $\Re(\widehat b(\lambda)) \ge 0$ for any $\Re \lambda> 0$ (see \cite{NohelShea} or \cite[page 38]{pruss}). Now it is clear that the positivity property \eqref{eq:1}  is not sufficient, in general, to ensure smoothing effects like \eqref{eq:regular} when working with
kernels that are more general than \eqref{eq:riesz}. This is why, following \cite{CDaPP} and \cite{MP97}, we will impose stronger conditions on $b$.
\begin{hypothesis}\label{hyp:b}
The kernel $0\neq b\in L^1_{loc}(\R_+)$, is $3$-monotone; that is, $b$, $-\dot b$ are nonnegative, nonincreasing, convex, and $\lim_{t\to \infty}b(t)=0$. Furthermore,
\begin{equation}\label{eq:sector}
\rho  := 1 + \frac{2}{\pi}\sup \{ | \mathrm{arg} \, \widehat b(\lambda) |, \; \Re\lambda >0 \} \in (1,2).
\end{equation}
\end{hypothesis}
In the special case of the Riesz kernel given in \eqref{eq:riesz} one can easily show that $\rho = 1+\beta$. From now on we set $\beta=\rho-1$
with $\rho$ defined by \eqref{eq:sector}.
\begin{remark}
It follows from \cite[Proposition 3.10]{pruss} that for 3-monotone and locally integrable kernels $b$, condition \eqref{eq:sector} is equivalent to
%(see \cite[Prop. 3.3, Chap. I]{pruss}).
\begin{equation}\label{eq:b-smooth}
\lim_{t\rightarrow 0} \frac{\frac 1t\int_0^t s b(s) \, ds}{\int_0^t -s \dot{b}(s) \, ds} < +\infty.
\end{equation}
Also note that, by \eqref{eq:sector}, we have that $ \Re (\widehat b(\lambda))\ge 0$ for $\Re\lambda>0$ and hence $b$ satisfies \eqref{eq:1}.
\end{remark}
For $A$ and $Q$ we suppose that there
exists numbers $\alpha>0$ and $ \kappa \in \mathbb{R}$ such that
\begin{equation}\label{eq:traceA}
\mathrm{Tr}(A^{-\alpha}) < +\infty,
\end{equation}
\noindent
\begin{equation}\label{eq:Q}
A^\kappa Q \in {\mathcal B}(H),\quad \alpha-\frac{1}{\rho}<\kappa\le \alpha.
\end{equation}

%We will assume also a technical assertion of $b$ which will used in order to derive some asymptotic rate of $\widehat b$:
%\begin{hypothesis}\label{hyp:b-anal}
%The kernel $b$ can be extended in a analytical way around the real axis.
%\end{hypothesis}
%
%HAS TO BE DETAILED

\subsection{The nonhomogeneous deterministic problem}

Given $f \in L^1([0,T];H)$, Assumption \ref{hyp:b} together with the fact that $A$ is positive and self-adjoint implies that the deterministic problem,
\begin{equation}\label{eq:detvolterra}
\dot u(t) + \int_0^t b(t-s) Au(s) \, ds = f(t), \quad t\in (0,T], \quad u(0)=u_0\in H,
\end{equation}
\noindent is well posed for all $T>0$. Indeed,  there exists a resolvent family $\{S(t)\}_{t\geq 0} \subset {\mathcal B}(H)$ which is strongly continuous for $t\ge 0$, differentiable for $t>0$ and uniformly bounded by 1, see \cite[Corollary 1.2 and Corollary 3.3]{pruss}. The unique mild solution of \eqref{eq:detvolterra} is given by the following variation of parameter formula \cite[Proposition 1.2]{pruss}
$$
u(t) = S(t)u_0 + \int_0^t S(t-s) f(s) \, ds, ~t\in [0,T].
$$
\begin{remark}\label{rem:existenceC1}
The positivity of the kernel $b$ defined in \eqref{eq:1}, together with the positivity of the operator $A$ already allows for the construction of a unique solution to \eqref{eq:detvolterra} using an energy argument, see \cite[Corollary 1.2]{pruss}. Assumption \ref{hyp:b} gives further integrability and smoothing properties for $\{S(t)\}_{t\ge 0}$.
\end{remark}

Note that such a resolvent family does not satisfy the semi-group property because of the non local feature of the memory term in \eqref{eq:detvolterra}. Nevertheless, it can be written explicitly using the spectral decomposition \eqref{eq:spectral} of $A$ as
\begin{equation}\label{eq:sk}
S(t) v = \sum_{k=1}^{+\infty} s_k(t) (v,e_k) e_k,
\end{equation}
\noindent where the functions $s_k(t)$ are the solutions of the ordinary differential equations
\begin{equation}\label{eq:skeq}
\dot{s_k}(t) + \lambda_k \int_0^t b(t-s) s_k(s)  \, ds = 0, \quad s_k(0)=1.
\end{equation}
The next proposition summarizes the main properties of the functions $\{s_k\}_{k\geq 1}$.
\begin{proposition}\label{prop:sk}
Suppose that $b$ satisfies Assumption \ref{hyp:b} and let $\rho\in (1,2)$ as defined in (\ref{eq:sector}). Then $\lim_{r\to\infty}s_k(r)=0$ for all $k\ge 1$ and there exists $C_0>0$ such that for any $k\geq 1$,
\begin{eqnarray}
\|s_k\|_{L^\infty(\R_+)} & \leq & 1,\label{eq:reg1}\\
\|\dot{s_k}\|_{L^1(\R_+)} & \leq & C_0,\label{eq:reg2}\\
\|t \dot{s_k}\|_{L^1(\R_+)} & \leq & C_0\, \lambda_k ^{-1/\rho},\label{eq:reg3}\\
\|s_k\|_{L^1(\R_+)} & \leq & C_0  \, \lambda_k ^{-1/\rho}.\label{eq:reg4}
\end{eqnarray}
\end{proposition}
\begin{proof}
Estimate \eqref{eq:reg1} follows from \cite[Corollary 1.2]{pruss}, inequalities (\ref{eq:reg2}) and (\ref{eq:reg3}) can be found in \cite[Proposition 6]{MP97} and estimate (\ref{eq:reg4}) is shown in \cite[Lemma 3.1]{CDaPP} where also the fact $\lim_{r\to\infty}s_k(r)=0$ for all $k\ge 1$ is shown in the proof of the lemma.
\end{proof}

Smoothing effects of the resolvent family $\{S(t)\}_{t\ge 0}$ when $b$ satisfies Assumption \ref{hyp:b} can be now easily proved using Proposition \ref{prop:sk}. %The proofs which derived easily from \cite{CDaPP} (Lemma 3.1, p. 213) are left to the reader.
\begin{proposition}\label{prop:smooth}
Let $b$ and $\rho$ as in Proposition \ref{prop:sk}. Then for any $t>0$, there exist a constants $C_0,C_1>0$ such that for any
$0\leq s \le 2/\rho$ and $0 \leq s' \le 2$,
\begin{eqnarray}
\| A^{s/2} S(t)\|_{{\mathcal B}(H)}   & \leq & C_0 \, t^{-s \rho/2}, \quad t>0, \label{eq:regS1} \\
&& \nonumber \\
\| A^{-s'/2} \dot{S}(t) \|_{{\mathcal B}(H)} & \leq & C_1\|b\|^{s'/2}_{L^1(0,t)} \, t^{s'/2 - 1}, \quad t>0. \label{eq:regS2}
\end{eqnarray}
\end{proposition}
\begin{proof}
For any $\delta \in (0,1)$ and any $k \geq 1$, H\"older's inequality, \eqref{eq:reg2} and \eqref{eq:reg3} yields
\begin{eqnarray*}
\int_0^{+\infty} u^\delta |\dot{s_k}(u)| \, \dd u & = & \int_0^{+\infty} u^\delta |\dot{s_k}(u)|^\delta |\dot{s_k}(u)|^{1-\delta} \, \dd u  \label{eq:regS1-1}\\
& \leq & \left (  \int_0^{+\infty} u |\dot{s_k}(u)| \, \dd u \right )^{\delta} \left ( \int_0^{+\infty} |\dot{s_k}(u)| \, \dd u  \right )^{1-\delta} \nonumber\\
&\leq & C_0 \; \lambda_k^{-\delta/\rho}.\nonumber
\end{eqnarray*}
Note, that the previous final estimate also holds for $\delta=0,1$ by \eqref{eq:reg2} and \eqref{eq:reg3}. Then, since $s_k(t) = - \int_t^{+\infty} u^{-\delta} u^{\delta} \dot{s_k}(u) \, du$ as $\lim_{r\to\infty}s_k(r)=0$ for all $k\ge 1$ by Proposition \ref{prop:sk}, we can conclude that
\begin{equation}\label{eq:regS1-2}
|s_k(t)| \leq C_0 \, t^{-\delta} \lambda_k^{-\delta/\rho}, \quad t>0,\quad \delta \in [0,1].
\end{equation}
Thus, for any $s \in [0,2/\rho]$ and $x \in H$, \eqref{eq:regS1-2} with $0\le \delta = \rho s/2 \le 1$ implies
\begin{eqnarray*}
\|A^{s/2} S(t) x \|^2 & = & \sum_{k\geq 1} \lambda_k^{s} \,  s_k(t)^2 (x,e_k)^2
 \leq  C_0 \; t^{-\rho s/2} \|x\|^2,
\end{eqnarray*}
which is \eqref{eq:regS1}. To show \eqref{eq:regS2}, we use \cite[Corollary 3.3]{pruss} which states that
under Assumption \ref{hyp:b} and since $0$ belongs to the resolvent set of $A$, there is $M>0$ such that
\begin{equation} \label{eq:regS2-1}
\| \dot{S}(t)x \| \leq M  t^{-1} \|x\|, \quad x \in H, \quad t>0.
\end{equation}
On the other hand, we can bound $\dot{S}(t)x$ for $x\in D(A)$ as follows:
\begin{align}\label{eq:regS2-2}
\| \dot{S}(t) x \|^2 & =  \sum_{k\ge 1} (\dot{s_k}(t))^2 (x,e_k)^2\\
& = \sum_{k \ge 1} \lambda_k^2 \left ( \int_0^t b(t-s) s_k(s) ds \right )^2 (x,e_k)^2 \leq  \|b\|_{L^1(0,t)}^2 \| A x \|^2,
\end{align}
where we have used \eqref{eq:skeq} and \eqref{eq:reg1}.
Finally, interpolation between \eqref{eq:regS2-1} and \eqref{eq:regS2-2} yields \eqref{eq:regS2}.
\end{proof}

\begin{remark}\label{rem:Sdot}
The estimate in \eqref{eq:regS2} does not provide an optimal rate, in fact it is the worst case scenario, as further smoothing may come from $\|b\|_{L^1(0,t)}$. The rate can be improved if we impose further regularity assumptions on $b$. Indeed, if in addition, $b$ satisfies Assumption \ref{hyp:b-anal} from Subsection \ref{sub:dete}, then by \eqref{eq:omega0} and \eqref{eq:omega1} it follows that $\hat{b}(\lambda)\sim \lambda^{1-\rho}$ as $\lambda\to \infty$. Thus, the nonnegativity of $b$ implies that $\|b\|_{L^1(0,t)}\le Ct^{\rho-1}$ by a Tauberian theorem for the Laplace transform (see, for example, \cite[Chapter V, Theorem 4.3]{Widder}). Therefore, in this case, we get a sharper estimate
$$
\| A^{-s'/2} \dot{S}(t) \|_{{\mathcal B}(H)}  \leq  C_1 \, t^{\rho s'/2 - 1}, \quad t>0,~0 \leq s' \le 2.
$$
Nevertheless, the rate in given \eqref{eq:regS2} is sufficient for our needs when it is used in the deterministic error analysis for smooth initial data.
\end{remark}

% 4 Rappels sur l'equation stochastique continue
\subsection{The continuous stochastic problem.}

Next we recall an existence result for the problem \eqref{eq:stovolterra} and, for the sake of completeness, we indicate a proof (see \cite[Theorem 2.1]{CDaPP} and we refer to \cite{sperlich} for more general noise).

\begin{proposition}\label{prop:existence}
Let $A$ and $Q$ satisfy  \eqref{eq:traceA}--\eqref{eq:Q} and let $b$ satisfy Assumption \ref{hyp:b}.  Then there exists an unique $H$-valued (Gaussian) weak solution $u$ of \eqref{eq:stovolterra} given by  the variation of constants formula
\begin{equation}\label{eq:stoconv}
u(t) = S(t) u_0 + \int_0^t S(t-s) \, dW^{_Q}(s).
\end{equation}
Furthermore, the stochastic convolution term has a version whose trajectories are a.s. $\theta$-H\"older continuous for any $\theta < ( 1 - \rho(\alpha - \kappa))/2$.
%$$
%\theta < \frac 12 \Big ( 1 - \rho(\alpha - \kappa) \Big ).
%$$
\end{proposition}
\begin{proof}
Analogously to \cite[Theorem 5.4]{DPZ}, it is sufficient to show that the stochastic convolution is well-defined. By It\^o's Isometry,
\begin{eqnarray*}
\mathbb{E}\left\|\int_0^t S(t-s) \, dW^{_Q}(s)\right\|^2&=&\int_0^t \| S(t-s) Q^{1/2} \|^2_{{\mathcal L}_2(H)} ds \\
&=&  \int_0^t \sum_{i\ge 1} \|S(t-s) Q^{1/2} e_i \|^2 \, ds \\
&=&\int_0^t \sum_{i,j\ge 1} (S(t-s) Q^{1/2} e_i,e_j)^2 \, ds \\
&=& \sum_{j\ge 1} \sum_{i\ge 1} \left ( \int_0^t s^2_j(t-s) \, ds \right ) (Q^{1/2} e_i, e_j)^2 \, ds \\
&\leq & C_0 \sum_{j\ge 1} \sum_{i\ge 1} \lambda_j^{-1/\rho} (Q^{1/2} e_i, e_j)^2\\
& = & C_0 \| A^{-1/(2\rho)} Q^{1/2}\|_{\mathcal{L}_2(H)}^2,
\end{eqnarray*}
\noindent where we have used Parseval's identity, \eqref{eq:reg1} and \eqref{eq:reg4}. By \eqref{eq:Q} we have that $-1/\rho - \kappa < -\alpha$, and thus using also \eqref{eq:trace_prop2},
\begin{align*}
\| A^{-1/(2\rho)} Q^{1/2} \|^2_{\mathcal{L}_2(H)} &=\Tr(A^{-1/\rho}Q)=\Tr(A^{-1/\rho - \kappa}A^{\kappa}Q)\\
&\le \mathrm{Tr}(A^{-1/\rho - \kappa}) \|A^{\kappa} Q\|^2_{\mathcal{B}(H)}
\le   \mathrm{Tr}(A^{-\alpha}) \|A^\kappa Q\|_{\mathcal{B}(H)}.
\end{align*}
Finally, the proof of the H\"older regularity in time of $u$ uses similar techniques and is omitted.
\end{proof}

\begin{remark}\label{rem:HSassumption}
Note that assumptions  \eqref{eq:traceA}--\eqref{eq:Q} are stronger than the minimal assumption $\| A^{-1/(2\rho)} Q^{1/2} \|_{\mathcal{L}^2(H)} < +\infty$ needed for the existence of a mean squared continuous solution. One can replace  \eqref{eq:traceA}--\eqref{eq:Q} by $$\| A^{(s-\frac{1}{\rho})/2} Q^{1/2} \|_{\mathcal{L}^2(H)} < +\infty$$ for some $s> 0$ as a single main assumption on $A$ and $Q$ and obtain H\"older regularity of order less than $\min(\frac12,\frac{\rho s}{2})$.
\end{remark}
%THE MINIMAL ASSUMPTIONS ARE AS USUAL $\| A^{-1/(2\rho)} Q^{1/2} \|_{HS} < \infty$. I INSIST HERE ON THE TIME REGULARITY WHICH SHOULD BE THE TIME ORDER ... OTHERWISE THE TRICK IS NOT OPTIMAL

%\begin{remark}
%Stochastic equations (\ref{eq:stovolterra}) has been studied in \cite{CDaPP} in which the authors show that under Assumption \ref{hyp:b}, (\ref{eq:stovolterra}) is well posed and  give a representation of the solution $u(t)$ as a stochastic convolution (\ref{eq:stoconv}).
%\end{remark}
\section{Space discretization} \label{sec:space}
%%%%%%%%%%%%%%%%%
%       SPACE PART
%%%%%%%%%%%%%%%%%
In this section we discretize \eqref{eq:stovolterra} in space by a standard piecewise continuous finite element method. We refer to the monograph \cite{Thomeebook} for further details on finite elements. We shall derive strong error estimates for the spatially semidiscrete problem for smooth initial data only imposing Assumption \ref{hyp:b} on $b$. We will see later that for time discretization and also for the fully discrete scheme, we have to put further restrictions on $b$. Let $\{\TT_h\}_{0<h<1}$ denote a family of triangulations of $\cD$, with mesh size $h>0$ and consider
   finite element spaces $\{ V_h \}_{0<h<1}$, where $V_h\subset H^1_0(\cD)$ consists of continuous piecewise linear functions vanishing at the boundary of $\DD$. In order to derive the finite element formulation we look for a $V_h$-valued process $u_h$ such that
\begin{equation*}
\left\{  \begin{aligned}
      &\inner{\dd u_h(t)}{\chi}+\int_0^tb(t-s)\inner{\nabla u_h(t)}{\nabla \chi}\,\dd s\,\dd t
       =\inner{\dd W^Q(t)}{\chi},~ \chi\in V_h, \ t>0,\\
       &\inner{u_h(0)}{\chi}=\inner{u_0}{\chi}.
  \end{aligned}
  \right.
\end{equation*}
We introduce the "discrete Laplacian"
\begin{equation}\label{def:Ah}
  A_{h}:V_h\to V_h,   \quad
    \inner{A_{h} \psi}{ \chi} = \inner{ \nabla \psi}{\nabla \chi},\quad \psi,\chi \in V_h,
\end{equation}
and the orthogonal projector
$$
    P_{h}: H \to V_h,\quad
   \inner{P_{h} f}{ \chi} = \inner{ f} {\chi},\quad \chi \in V_h.
$$
It is clear that the operator $A_h$ is a positive definite bounded operator on $V_h$. %Moreover, it can be seen that by H\"older inequalities, $A_h$ satisfies the following
%uniform interpolation inequality:
%\begin{equation}\label{eq:holder-h}
%\| A_h^s \chi \| \leq \| A_h^{s_1} \chi \|^\lambda \| A_h^{s_2} \chi \|^{1-\lambda}, \quad \chi \in V_h, \quad s = \lambda s_1+ (1-\lambda)s_2.
%\end{equation}
Let us note also that using the definition \eqref{def:Ah} of $A_h$, the following uniform inequality can be easily derived
\begin{equation}\label{eq:33}
\|A_h^{-1/2} P_h x \| \leq \| A^{-1/2} x\|, \quad x\in H.
\end{equation}
Then, using the $L_2$-stability of $P_h$ and some interpolation theory, we also have that
\begin{equation}\label{eq:Ah}
\|A_h^{-\delta}P_hx\|\le \|A^{-\delta}x\|,\quad \delta\in [0,\frac{1}{2}], \quad x \in H.
\end{equation}
Similarly to $-A$, the operator $-A_h$ generates an analytic contraction semigroup on $V_h$ and satisfies the uniform resolvent estimate
\begin{equation*}
\|z(z+A_h)^{-1}P_h\|=\|z R(z,A_h)P_h\|\le M_{\phi},
\end{equation*}
for $z \in \Sigma_{\phi}=\{z\in \mathbb{C}: |\arg(z)|<\phi<\pi\}.$
Since $A_h R(z,A_h) = I - zR(z,A_h)$, it follows that
\begin{equation}\label{eq:resolvAh}
\| A_h R(z,A_h)P_h \|_{{\mathcal B}(H)} \leq M_\phi+1, \quad z \in \Sigma_\phi.
\end{equation}
Then we can rewrite the spatially semidiscrete problem in the same form as the original one as
\begin{equation}\label{semid}
\left\{
  \begin{aligned}
  &\dd u_h + \left(\int_0^tb(t-s)A_h u_h(s) \,\dd s\right)\,\dd t = P_h\,\dd W^Q(t),\quad t>0, \\
  &u_h(0)=P_hu_0.
\end{aligned}
\right.
\end{equation}
Similarly to the original problem the weak solution is given by
$$
u_h(t)=S_h(t)P_hu_0+\int_0^tS_h(t-s)P_h\,\dd W^Q(s),
$$
where the resolvent family $\{S_h(t)\}_{t\ge 0}$ can be written explicitly as
$$
S_h(t)P_h u_0=\sum_{k=1}^{\infty}s_{h,k}(t)( u_0, e_{h,k}) e_{h,k}.
$$
Here $(\lambda_{h,k},e_{h,k})$ are the eigenpairs of $A_h$ and $s_{h,k}(t)$ are the solution of the ODEs
$$
\dot{s}_{h,k}(t)+\lambda_{h,k}\int_0^t b(t-s)s_{h,k}(s)\,\dd s =0, ~s_{h,k}(0)=1.
$$
We have the following stability result.
\begin{lemma}\label{lem:stab}
If $b$ satisfies Assumption \ref{hyp:b}, then for some $C>0$,
$$
\int_0^t\|S(s)x\|^2\,\dd s\le C\|x\|^2_{-\frac{1}{\rho}},~t>0,
$$
and
$$
\int_0^t\|S_h(s)P_hx\|^2\,\dd s\le C\|x\|^2_{-\frac{1}{\rho}},~t>0,~h>0.
$$
\end{lemma}
\begin{proof}
We have, by \eqref{eq:reg1} and \eqref{eq:reg4}, that
\begin{multline*}
\int_0^t\|S(s)x\|^2\,\dd s=\sum_{k=1}^{\infty}\int_0^ts^2_k(s)\,\dd s\, \inner{x}{e_k}^2\\
\le
\sum_{k=1}^{\infty}\|s_k\|_{L^{\infty}(\R_+)}\|s_k\|_{L^1(\R_+)}\inner{x}{e_k}^2
\le
C_0 \sum_{k=1}^{\infty} \lambda_k^{-1/\rho} \inner{x}{e_k}^2
=C_0\|x\|^2_{-\frac{1}{\rho}}.
\end{multline*}
As the constants in \eqref{eq:reg1} and \eqref{eq:reg4} do not depend on $\lambda_k$, we similarly obtain
$$
\int_0^t\|S_h(s)P_hx\|^2\,\dd s\le C_0\|A_h^{-1/2\rho}P_hx\|^2.
$$
Finally, since $-1/2<-1/2\rho<-1/4$, using (\ref{eq:Ah}) with $\delta = 1/(2\rho)$, completes the proof.
\end{proof}

The error analysis is based on the Ritz projection
\begin{equation*}
R_h:H^1_0(\cD)\to V_h,\quad \inner{\nabla R_h v}{\nabla \chi} =\inner{\nabla v}{\nabla \chi},~v\in H^1_0(\cD),~\chi\in V_h.
\end{equation*}
In particular, we assume that $R_h$ satisfies the error bound
\begin{equation}\label{eq:Ritz}
\|R_hv-v\|\le Ch^{\gamma}\|v\|_{\gamma},~v\in D(A^{\gamma/2}),~1\le \gamma\le 2.
\end{equation}
This puts some restriction on the domain $\cD$ but it is satisfied for convex polygonal domains, for instance.

Next we prove an $L^2((0,\infty),H)$ error estimate for the space semidiscretization of the deterministic problem. It is an extension of the result in \cite{choi} where the special kernel $b(t)=\frac{1}{\Gamma(\beta)}e^{-t}t^{\beta-1}$ was considered.
\begin{proposition}\label{prop:semid}
If $b$ satisfies Assumption \ref{hyp:b} and \eqref{eq:Ritz} holds,  then
$$
\int_0^{\infty}\|S(t)x-S_h(t)P_hx\|^2\,\dd t \le Ch^{2s}\|x\|^2_{s-\frac{1}{\rho}},~0\le s\le 2.
$$
\end{proposition}
\begin{proof}
It follows from Lemma \ref{lem:stab} that
\begin{equation}\label{eq:stab}
\int_0^t\|(S(s)-S_h(s)P_h)x\|^2\,\dd s\le 2 \int_0^t\|S(s)x\|^2+\|S_h(s)P_hx\|^2\,\dd s \le C\|x\|^2_{-\frac{1}{\rho}}.
\end{equation}
To prove an error estimate of optimal order we set
\begin{multline*}
e(t):=S(t)x-S_h(t)P_hx:=v(t)-v_h(t)\\=
v(t)-P_hv(t)+P_hv(t)-v_h(t):=\rho(t)+\theta(t).
\end{multline*}
For $\rho$, using the best approximation property of $P_h$, we obtain by Lemma \ref{lem:stab} and \eqref{eq:Ritz} that
\begin{equation}\label{eq:rho}
\int_0^\infty\|\rho(t)\|^2\,\dd t\le \int_0^\infty \|(R_h-I)v(t)\|^2\,\dd t\le Ch^4 \|x\|^2_{2-\frac{1}{\rho}}.
\end{equation}
In a standard way one derives an equation for $\theta$ which reads
\begin{equation*}
\left\{
\begin{aligned}
&\dot{\theta}(t)+\int_0^tb(t-s)A_h\theta(s)\,\dd s=A_hP_h\int_0^t b(t-s)(R_h-I)v(s)\,\dd s,~t>0,\\
&\theta(0)=0.
\end{aligned}
\right.
\end{equation*}
Taking Laplace transforms of both sides yields
$$
z \widehat{\theta}(z)+\widehat{b}(z)A_h\widehat{\theta}(z)=A_hP_h(R_h-I)\widehat{v}(z)\widehat{b}(z).
$$
Therefore,
\begin{equation}\label{eq:theta}
\widehat{\theta}(z)=A_hR(\frac{z}{\widehat{b}(z)},A_h)P_h(R_h-I)\widehat{v}(z).
\end{equation}
It can be shown that $\widehat{b}$ extends continuously to $i\mathbb{R}\setminus \{0\}$, see, for example, \cite{MP97}. Therefore, using \eqref{eq:sector}, it follows that $\frac{ik}{\widehat{b}(ik)}\in \Sigma_{\phi}$, $k\in \R\setminus \{0\}$, with $\phi<\pi$. Thus, $\|A_hR(\frac{ik}{\hat{b}(ik)},A_h)P_h\|_{\mathcal{B}(H)}\le (M_\phi+1)$ by \eqref{eq:resolvAh}. Therefore, setting $z=ik$, $k\in \R\setminus \{0\}$, in \eqref{eq:theta} and using
the isometry property of the Fourier transform we obtain, by Lemma \ref{lem:stab} and \eqref{eq:Ritz}, that
\begin{equation}\label{eq:errsp}
\int_0^\infty\|\theta(t)\|^2\,\dd t\le (M_\phi+1) \int_0^\infty \|(R_h-I)v(t)\|^2\,\dd t\le Ch^4 \|x\|^2_{2-\frac{1}{\rho}}.
\end{equation}
Interpolation using \eqref{eq:stab}, \eqref{eq:rho}, and \eqref{eq:errsp} yields
\begin{equation*}
\int_0^{\infty}\|e(t)\|^2\,\dd t \le 2\int_0^\infty\left(\|\rho(t)\|^2+\|\theta(t)\|^2\right)\,\dd t\le Ch^{2s}\|x\|^2_{s-\frac{1}{\rho}},~0\le s\le 2.
\end{equation*}
\end{proof}

Next, using the error analysis from \cite{McLThomee93}, we have the following pointwise smooth data estimate for the spatially semidiscrete scheme.
\begin{proposition}\label{prop:pw}
If $b$ satisfies Assumption \ref{hyp:b} and \eqref{eq:Ritz} holds,  then for every $\epsilon>0$ there is $C=C(T,\epsilon)$ such that
$$
\|S(t)x-S_h(t)P_hx\| \le Ch^{s}\|x\|_{s(1+\epsilon)},~0\le s\le 2,~t\in [0,T].
$$
\end{proposition}
\begin{proof}
As already observed, Assumption \ref{hyp:b} implies that $b$ is a positive definite kernel. Therefore by, \cite[Theorem 2.1]{McLThomee93}, it follows that
$$
\|S(t)x-S_h(t)P_hx\|\le Ch^2\left(\|x\|_2+\int_0^t\|\dot{S}(s)x\|_2\,\dd s\right).
$$
Proposition \ref{prop:smooth} implies that
\begin{equation}\label{eq:sss}
\int_0^t\|\dot{S}(s)x\|_2\,\dd s
=\int_0^t\|A^{-\epsilon}\dot{S}(s)A^{1+\epsilon}x\|\,\dd s\le C(T,\epsilon)\|x\|_{2+2\epsilon}.
\end{equation}
Finally, since $\|S(t)-S_h(t)P_h\|_{\mathcal{B}(H)}\le 2$, interpolation finishes the proof.
\end{proof}

%\begin{theorem}
%Let $A$ and $Q$ satisfy (\ref{eq:traceA})--(\ref{eq:Q}) and let $b$ be satisfy Assumption \ref{hyp:b}.
%If $u_0=0$ for simplicity, then
%$$
%\left(\mathbb{E}\|u(t)-u_h(t)\|^2\right)^{1/2}\le C(\nu)h^\nu,~\nu\le\min(2,\frac{1}{\rho}-\alpha+\kappa).
%$$
%In particular, if $Q=I$ then $d=1$, $\kappa=0$  and $\alpha>\frac12$ and hence $\nu< \frac{1}{\rho}-\frac{1}{2}$.
%\end{theorem}
%\begin{proof}
%Using the variation of constants formula,
%$$
%u(t)-u_h(t)=\int_0^t(S(t-s)-S_h(t-s)P_h)Q^{1/2}\,\dd W(s)
%$$
%and hence, by It\^o's Isometry and Proposition \ref{prop:semid},
%\begin{align*}
%\mathbb{E}\|u(t)-u_h(t)\|^2&=\mathbb{E}\left\|\int_0^t(S(t-s)-S_h(t-s)P_h)Q^{1/2}\,\dd W(s)\right\|^2\\
%&=\int_0^t\|(S(t-s)-S_h(t-s)P_h)Q^{1/2}\|^2_{\HS}\,\dd s\\
%&=\sum_{k=1}^\infty \int_0^t\|(S(s)-S_h(s)P_h)Q^{1/2}\phi_k\|^2\,\dd s\\
%&\le C h^{2s} \sum_{k=1}^{\infty}\|A^{(s-\frac{1}{\rho})/2}Q^{1/2}\phi_k\|^2=Ch^{2s}\|A^{(s-\frac{1}{\rho})/2}Q^{1/2}\|_{\HS}^2\\
%&=Ch^{2s}\Tr(A^{s-\frac{1}{\rho}}Q)\le Ch^{2s}\Tr(A^{s-\frac{1}{\rho}-\kappa})\|A^{\kappa}Q\|.
%\end{align*}
%where $\{\phi_k\}$ is an ONB of $H$.
%\end{proof}

\begin{theorem}
Let $A$ and $Q$ satisfy (\ref{eq:traceA})--(\ref{eq:Q}) and let $b$ be satisfy Assumption \ref{hyp:b}.
If $\E\|u_0\|_{\nu(1+\epsilon)}^2<\infty$ and \eqref{eq:Ritz} holds, then there is $C=C(T,\epsilon,\nu)$ such that
$$
\left(\mathbb{E}\|u(t)-u_h(t)\|^2\right)^{1/2}\le Ch^\nu,~\nu\le\frac{1}{\rho}-\alpha+\kappa,~t\in [0,T].
$$
\end{theorem}
\begin{proof}
By the variation of constants formula,
$$
u(t)-u_h(t)=S(t)x-S_h(t)x+\int_0^t(S(t-s)-S_h(t-s)P_h)\,\dd W^Q(s).
$$
Thus,
\begin{align*}
\E\| u(t)-u_h(t)\|^2 &\le 2 \E\|S(t)x-S_h(t)x\|^2\\
&+2\E\left\|\int_0^t(S(t-s)-S_h(t-s)P_h)\,\dd W^Q(s)\right\|^2:=e_1+e_2.
\end{align*}
It follows from Proposition \ref{prop:pw} that
$$
e_1\le Ch^{2\nu}\E\|u_0\|^2_{\nu(1+\epsilon)}.
$$
To bound $e_2$ we use It\^o's Isometry and Proposition \ref{prop:semid} to obtain
\begin{equation}\label{eq:e2}
\begin{aligned}
e_2&=\mathbb{E}\left\|\int_0^t(S(t-s)-S_h(t-s)P_h)\,\dd W^Q(s)\right\|^2\\
&=\int_0^t\|(S(t-s)-S_h(t-s)P_h)Q^{1/2}\|^2_{{\mathcal L}_2(H)}\,\dd s\\
&=\sum_{k=1}^\infty \int_0^t\|(S(s)-S_h(s)P_h)Q^{1/2}e_k\|^2\,\dd s\\
&\le C h^{2\nu} \sum_{k=1}^{\infty}\|A^{(\nu-\frac{1}{\rho})/2}Q^{1/2}e_k\|^2=Ch^{2\nu}\|A^{(\nu-\frac{1}{\rho})/2}Q^{1/2}\|_{{\mathcal L}_2(H)}^2\\
&=Ch^{2\nu}\Tr(A^{\nu-\frac{1}{\rho}}Q)\le Ch^{2\nu}\Tr(A^{\nu-\frac{1}{\rho}-\kappa})\|A^{\kappa}Q\|.
\end{aligned}
\end{equation}
\end{proof}
\begin{remark}\label{rem:spaceaq}
In particular, if $Q=I$, then $d=1$, $\kappa=0$  and $\alpha>\frac12$ whence $\nu< \frac{1}{\rho}-\frac{1}{2}$. Also note, that it is clear from the proof that instead of (\ref{eq:traceA})--(\ref{eq:Q}) we could assume that $\|A^{(\nu-\frac{1}{\rho})/2}Q^{1/2}\|_{{\mathcal L}_2(H)}<\infty$ and get a convergence rate of order $\nu$. Then, for trace class noise; that is, when $\Tr(Q)<\infty$ we can take
$\nu=\frac{1}{\rho}$.
\end{remark}
We end this section by showing that the above error estimate is optimal in the sense that it corresponds to the spatial regularity of the solution.
\begin{theorem}
Let $A$ and $Q$ satisfy (\ref{eq:traceA})--(\ref{eq:Q}) and let $\nu=\frac{1}{\rho}-\alpha+\kappa$, or, let $\|A^{(\nu-\frac{1}{\rho})/2}Q^{1/2}\|_{\mathcal{L}_2(H)}<\infty$ for some $\nu\ge 0$. If $b$ satisfies Assumption \ref{hyp:b} and $\E\|u_0\|^2_{\nu}<\infty$, then $\E\|u(t)\|^2_{\nu}\le C$ for some $C>0$ for all $t\ge 0$.
\end{theorem}
\begin{proof}
It follows by It\^{o}'s Isometry and the fact that $\|S(t)\|\le 1$ that
$$
\E\|u(t)\|_{\nu}^2\le 2\E\|u_0\|_\nu^2+2\int_0^t\|A^{\nu/2}S(s)Q^{1/2}\|^2_{{\mathcal L}_2(H)}\dd s.
$$
Let $(e_k,\lambda_k)$ be the eigenpairs of $A$. Then, by monotone convergence, the self-adjointness of $A$ and $S$, and Proposition \ref{prop:sk}, it follows that
\begin{align*}
&\int_0^t\|A^{\nu/2}S(s)Q^{1/2}\|^2_{{\mathcal L}_2(H)}\dd s=\sum_{k=1}^{\infty}\int_0^t\|A^{\nu/2}S(s)Q^{1/2}e_k\|^2\,\dd s \\
&=\sum_{j,k=1}^{\infty}\int_0^t (A^{\nu/2}S(s)Q^{1/2}e_k,e_j)^2\,\dd s = \sum_{j,k=1}^{\infty}\int_0^t (Q^{1/2}e_k,S(s)A^{\nu/2}e_j)^2\,\dd s\\
&=\sum_{j,k=1}^{\infty}(Q^{1/2}e_k,\lambda_j^{n/2}e_j)^2 \int_0^t s^2_j(s)\,\dd s\\
&\le \sum_{j,k=1}^{\infty}(Q^{1/2}e_k,\lambda_j^{\nu/2}e_j)^2
\|s_j\|_{L^{\infty}(\R_+)}\|s_j\|_{L^1(\R_+)} \\
&\le C_0 \sum_{j,k=1}^{\infty}(Q^{1/2}e_k,\lambda_j^{\nu/2}e_j)^2\lambda_j ^{-1/\rho}=C_0\sum_{j,k=1}^{\infty}(Q^{1/2}e_k,\lambda_j^{\nu/2-\frac{1}{2\rho}}e_j)^2\\
&=C_0\|A^{(\nu-\frac{1}{\rho})/2}Q^{1/2}\|^2_{{\mathcal L}_2(H)}\le C_0\Tr(A^{\nu-\frac{1}{\rho}-\kappa})\|A^{\kappa}Q\|^2_{\mathcal{B}(H)}.
\end{align*}
\end{proof}

%%%%%%%%%%%%%%%%%%%%%%%%%
%   III.  TIME PART
%%%%%%%%%%%%%%%%%%%%%%%%%
\section{Time discretization} \label{sec:time}

Time discretization is achieved via a classical implicit Euler scheme and, concerning the convolution in time, via a
quadrature rule based on \eqref{eq:omega}. Let $\Delta t >0$ and
we set $t_n = n \, \Delta t$ for any integer $n\ge 0$.
We seek for an approximation
$u_n$ of $u(t_n)$ defined by the recurrence
\begin{equation} \label{eq:time_scheme}
u_n - u_{n-1} + \Delta t \left ( \sum_{k=1}^{n} \omega_{n-k}\,  A u_k \right ) = W^{_Q}(t_n) - W^{_Q}(t_{n-1}), \quad n\geq 1,
\end{equation}
\noindent with initial condition $u_0=u(0)$. We recall that the coefficients $\{\omega_k\}_{k\geq 0}$ of the quadrature are chosen such that
\begin{equation}\label{eq:omegab}
\sum_{k= 0}^{+\infty} \omega_k z^k  = \widehat b \left ( \frac{1-z}{\Delta t} \right ), \quad |z|<1.
\end{equation}
Let us note that thanks to \cite[estimate (3.6)]{MP97}, we have the lower bound for $\omega_0$:
\begin{equation}\label{eq:omega0}
\omega_0 = \widehat b(1/\Delta t) \geq c \Delta t^{\rho-1},~ \Delta t < 1,
\end{equation}
\noindent where $\rho\in(1,2)$ is defined in \eqref{eq:sector}.

 In the sequel we derive a discrete mild formulation (variation of constants formula) for \eqref{eq:time_scheme}. This formulation can not be made easily explicit as a function of the operators $A$, $Q$ and the kernel $b$, because of the memory effect in the drift. First consider the deterministic algorithm
\begin{equation} \label{eq:time_scheme det}
v_n - v_{n-1} + \Delta t \left ( \sum_{k=1}^{n} \omega_{n-k}\,  A v_k \right ) = 0, \quad n\geq 1; \quad v_0=x.
\end{equation}
Taking the $z$-transform, using the notation $$\hat{V}(z)=\sum_{k=0}^\infty v_kz^k\text{ and }\hat{\omega}(z)=\sum_{k=0}^{\infty}\omega_kz^k, $$   we get
$$
\hat{V}(z)-x-z\hat{V}(z)+\Delta t \hat{\omega}(z)A(\hat{V}(z)-x)=0.
$$
Thus,
$$
\hat{V}(z)=(I+\Delta t \hat{\omega}(z)A)((1-z)I+\Delta t \hat{\omega}(z)A)^{-1}x:=\hat{B}(z)x,
$$
where
$$
\hat{B}(z)x=\sum_{k=0}^\infty B_kx z^k.
$$
This means that $v_k=B_kx$, $k=0,1,...$ Note that $B_0=\hat{B}(0)=I.$ For the stochastic equation it will be useful to rewrite $\hat{B}(z)x$ as
\begin{equation}\label{eq:rew}
\begin{aligned}
\hat{B}(z)x=&((1-z)I+\hat{\omega}(z)\Delta t A)^{-1}(I+\hat{\omega}(z)\Delta t A)x\\
&=((1-z)I+\hat{\omega}(z)\Delta t A)^{-1}x+\hat{\omega}(z)\Delta t A((1-z)I+\hat{\omega}(z)\Delta t A)^{-1}x\\
&=((1-z)I+\hat{\omega}(z)\Delta t A)^{-1}x-(1-z)((1-z)I+\hat{\omega}(z)\Delta t A)^{-1}x+x\\
&=(z((1-z)I+\hat{\omega}(z)\Delta t A)^{-1}+I)x.
\end{aligned}
\end{equation}
Now, we consider the stochastic case \eqref{eq:time_scheme} which reads, after taking the $z$-transform, rearranging, and using the notation $w_n=W^Q(t_n)-W^Q(t_{n-1})$ for $n\ge 1$, $w_0=0$, and
$$
\hat{w}(z)=\sum_{k=0}^{\infty}w_kz^n\text{ and }\hat{U}(z)=\sum_{k=0}^{\infty}u_kz^k,
$$
as
\begin{align*}
\hat{U}(z)&=\hat{B}(z)x+((1-z)I+\hat{\omega}(z)\Delta t A)^{-1}\hat{w}(z)\\
&=\hat{B}(z)x+\frac{\hat{B}(z)-I}{z}\hat{w}(z)=\hat{B}(z)x+\hat{B}(z)\frac{\hat{w}(z)}{z}-\frac{1}{z}\hat{w}(z),
\end{align*}
where we also used \eqref{eq:rew} to rewrite the stochastic term in the previous calculation. This yields the discrete variation of constants formula, taking into account that $w_0=0$ and that $B_0=I$,
\begin{equation}\label{eq:distvar}
u_n=B_nx+\sum_{k=0}^nB_{n-k}w_{k+1}-w_{n+1}=B_nx+\sum_{k=0}^{n-1}B_{n-k}w_{k+1}.
\end{equation}
The importance of this formula lies in the fact that it connects the deterministic case to the stochastic case with the deterministic time-discrete solution operators  $B_n$ explicitly appearing in the formula.
\subsection{Deterministic estimates: stability and smoothing.} The next theorem is interesting in its own right. It shows that Lubich's convolution quadrature based on the Backward Euler scheme have a remarkable qualitative property: it preserves the $L^p$-norm of the orbits of the solution. The result can be viewed as a generalization of the ones in \cite{HaNo2011}; in particular, it removes the additional technical frequency condition in \cite[Theorem 2]{HaNo2011}.  The proof uses a representation similar to that in \cite{Pal}. We also note that the statement holds in Banach spaces as well since the proof does not use Hilbert space techniques.
\begin{theorem}\label{thm:lc}
If the resolvent family $\{S(t)\}_{t\ge 0}$ of \eqref{eq:detvolterra} satisfies $$S(\cdot)x\in L^p((0,\infty);H)$$ for some $1\le p \le \infty$ and $x\in H$, then
$$
\Delta t \sum_{k=1}^n\|B_kx\|^p\le \int_0^{\infty}\|S(t)x\|^p\,\dd t,\quad 1\le p<\infty,
$$
and
$$
\sup_{k\ge 1}\|B_kx\|\le \|S(\cdot)x\|_{L^{\infty}(\R_+)}.
$$
\end{theorem}
\begin{proof}
The Laplace Transform of $\{S(t)\}_{t\ge 0}$ is given by
$$
\hat{S}(z)x=(zI+\hat{b}(z)A)^{-1}x.
$$
Using \eqref{eq:omegab} and \eqref{eq:rew} we see that the $z$-transform $\hat{B}x$ of $\{B_nx\}_n$ is given by
\begin{equation*}
\begin{aligned}
\hat{B}(z)&=z\frac{1}{\Delta t}\hat{S}(\frac{1-z}{\Delta t})x +x=x+z\int_0^{\infty}S(\Delta t s)e^{-s}e^{zs}\,\dd s\\
&=x+\sum_{k=1}^\infty z^k\int_0^\infty S(\Delta t s)x\frac{e^{-s}s^{k-1}}{(k-1)!}\,\dd s.
\end{aligned}
\end{equation*}
Therefore, we conclude that $B_0=I$ and that
\begin{equation}\label{eq:bn}
B_kx=\int_0^\infty S(\Delta t s)x\frac{e^{-s}s^{k-1}}{(k-1)!}\,\dd s\text{ for }k\ge 1.
\end{equation}
Let
$$
f_k(s):=\frac{e^{-s}s^{k-1}}{(k-1)!},\quad k\ge 1.
$$
Then $f_k\ge 0$, $\|f_k\|_{L^1(\R_+)}=1$. Therefore, if $p=\infty$, we immediately obtain from \eqref{eq:bn} that
$$
\sup_{k\ge 1}\|B_kx\|\le \|S(\cdot)x\|_{L^{\infty}(\R_+)}.
$$
If $1\le p<\infty$, then we use Jensen's inequality in \eqref{eq:bn}, and have
\begin{multline*}
\Delta t \sum_{k=1}^n\|B_kx\|^p\le \sum_{k=1}^n\Delta t\int_0^{\infty} \|S(\Delta t s)x\|^pf_k(s)\,\dd s\\
=\int_0^{\infty} \|S( t)x\|^p\sum_{k=1}^nf_k(\frac{t}{\Delta t})\,\dd t
 \le \sup_{t> 0} \sum_{k=1}^{\infty}f_k(t)\int_0^{\infty} \|S( t)x\|^p\,\dd t.
\end{multline*}
%Finally, by monotone convergence, the Laplace transform of $\sum_{k=1}^\infty f_k$ is given by
%$$
%\widehat{(\sum_{k=1}^\infty f_k)}(\lambda)=\sum_{n=1}^\infty \widehat{f_k}(\lambda)=\sum_{k=1}^\infty\left(\frac{1}{1+\lambda}\right)^k=\frac{1}{\lambda},~\lambda>0.
%$$
Finally, noticing that $\sum_{n=1}^\infty f_n\equiv 1$ completes the proof.
\end{proof}
Theorem \ref{thm:lc} has the following important corollary on the smoothing and stability of the time discretization scheme in case $b$ satisfies Assumption 1.
\begin{corollary}\label{cor:smooth}
If $b$ satisfies Assumption \ref{hyp:b}, then, for all $x\in H$,
$$
\sup_{k\ge 1}\|B_kx\|\le \|x\|\text{ and }\Delta t \sum_{k=1}^n\|B_kx\|^2\le C\|x\|^2_{-\frac{1}{\rho}},~n\ge 1.
$$
\end{corollary}
\begin{proof}
The statement follows from Theorem \ref{thm:lc} together with Lemma \ref{lem:stab} and the fact that $\|S(t)\|\le 1$ for $t\ge 0$.
\end{proof}
Finally we will need a H\"older type estimate on the resolvent family $\{S(t)\}_{t\ge 0}$.
\begin{lemma}\label{lem:hol}
If $b$ satisfies Assumption \ref{hyp:b}, then there is $C=C(T,\gamma)>0$ such that
$$
\left(\sum_{k=1}^{n} \int_{t_{k-1}}^{t_{k}} \| (S(t_n - s) - S(t_n-t_{k-1}) ) x \|^2 \, \dd s\right)^{1/2}\le C\Delta t^\gamma \|x\|_{s-\frac{1}{\rho}},\,\, n\Delta t =T,
$$
for all $\gamma<\frac{\rho s}{2}$ where $0< s\le \frac{1}{\rho}$.
\end{lemma}
\begin{proof}
It follows from \eqref{eq:regS1}, with $s=\frac{1}{\rho}-\epsilon$, and Lemma \ref{lem:stab} that there is a constant $C=C(\epsilon, T)$ such that, for $0<\epsilon\le \frac{1}{\rho}$,
$$
\left(\sum_{k=1}^{n} \int_{t_{k-1}}^{t_{k}} \| (S(t_n - s) - S(t_n-t_{k-1}) ) x \|^2 \, \dd s\right)^{1/2}\le C \|x\|_{\epsilon-\frac{1}{\rho}},\,n\Delta t=t_n =T.
$$
Next, it follows from Proposition \ref{prop:sk} that
\begin{align*}
\sum_{k=1}^n\int_{t_{k-1}}^{t_{k}} &\| (S(t_n - s) - S(t_n-t_{k-1}) ) x \|^2\,\dd s\\
&=\sum_{i=1}^{\infty}( x, e_i) ^2\sum_{k=1}^n\int_{t_{k-1}}^{t_k}(s_i(t_n - s) - s_i(t_n-t_{k-1})^2\,\dd s\\
&\le 2 \sum_{i=1}^{\infty}( x, e_i) ^2\sum_{k=1}^n\int_{t_{k-1}}^{t_k}|s_i(t_n - s) - s_i(t_n-t_{k-1})|\,\dd s\\
&\le 2\sum_{i=1}^{\infty}( x, e_i) ^2\sum_{k=1}^n\int_{t_{k-1}}^{t_k}\int_{t_n-s}^{t_n-t_{k-1}}|\dot{s}_i(t)|\,\dd t\,\dd s\\
&\le 2\sum_{i=1}^{\infty}( x, e_i) ^2\sum_{k=1}^n\int_{t_{k-1}}^{t_k}\int_{t_n-t_k}^{t_n-t_{k-1}}|\dot{s}_i(t)|\,\dd t\,\dd s\\
%&\le 2\sum_{i=1}^{\infty}( x, e_i) ^2\sum_{k=1}^n\int_{t_{k-1}}^{t_k}\int_{t_{k-1}}^{t_k}|\dot{s}_i(t)|\,\dd s\,\dd t\\
&\le 2 \Delta t \|x\|^2\sup_{i\ge 1}\|\dot{s}_i\|_{L^1(\R_+)}\le C\Delta t \|x\|^2.
\end{align*}
Finally, interpolation gives the desired result.
\end{proof}
\subsection{Deterministic estimates: convergence rates}\label{sub:dete}
In order to give an error estimate of optimal order with no initial regularity for the time discretization of the deterministic problem we  have to impose another assumption on $b$. This kind of assumption; that is, the existence of an analytic extension of $\hat{b}$ to a sector beyond the left halfplane, is fairly standard in the existing deterministic literature, see, for example, \cite{eggermont1992,lubich88,lubich04,Lubich_et_al1996},  but it clearly represents a major restriction compared to Assumption \ref{hyp:b}. We note that this additional assumption is not needed neither for the spatial error estimates with smooth initial data, and hence for the space-semidiscretization of the stochastic equation, nor the for the stability results for the time discretization in the previous subsection.
\begin{hypothesis}\label{hyp:b-anal}
The Laplace transform $\widehat{b}$ of $b$ can be extended to an analytic function in a sector $\Sigma_\theta$ with $\theta>\pi/2$ and $|\widehat{b}^{(k)}(z)|\le C|z|^{1-\rho-k}$, $k=0,1$, $z\in \Sigma_\theta$.
\end{hypothesis}
Note that Assumption \ref{hyp:b-anal} implies that
\begin{equation}\label{eq:omega1}
\omega_0 = \widehat b(1/\Delta t) \leq C \Delta t^{\rho-1},~ \Delta t < 1.
\end{equation}
An important example of a family of kernels satisfying both Assumptions \ref{hyp:b} and \ref{hyp:b-anal} is given by $b(t)=Ct^{\beta-1}e^{-\eta t}$, $0<\beta<1$ and $\eta\ge 0$.

Assumptions \ref{hyp:b} and \ref{hyp:b-anal} allows us to use the following deterministic nonsmooth data estimate \cite[Theorem 3.2]{Lubich_et_al1996}.
\begin{proposition}\label{lem:LST96}
If Assumptions \ref{hyp:b} and \ref{hyp:b-anal} hold, then there exists $C=C(\rho)>0$ such that
\begin{equation}\label{eq:nonsmooth}
\| S(t_n)x - B_nx \| \leq \frac{C}{t_n}\Delta t \,\|x\|, \quad n\geq 1.
\end{equation}
\end{proposition}
%\begin{remark}\label{rem:regularization}
%Estimate (\ref{eq:nonsmooth}) means essentially that the deterministic error is of order 1 for initial data in $H$ but only after some while $t_n\geq 1$.
%\end{remark}

\begin{corollary}\label{cor:tee}
If Assumptions \ref{hyp:b} and \ref{hyp:b-anal} hold, then there exists $C=C(T,\gamma,\rho)$ such that
$$
\left(\Delta t\sum_{k=0}^n\|S(t_k)x-B_kx\|^2\right)^{1/2}\le C\Delta t^\gamma \|x\|_{s-\frac{1}{\rho}},\quad n\Delta t =T,
$$
for all $\gamma<\frac{\rho s}{2}$ where $0< s\le \frac{1}{\rho}$.
\end{corollary}
\begin{proof}
It follows from \eqref{eq:regS1}, with $s=\frac{1}{\rho}-\epsilon$, and Corollary \ref{cor:smooth} that there is a constant $C=C(\epsilon, T)$ such that, for $0<\epsilon\le \frac{1}{\rho}$,

$$
\left(\Delta t\sum_{k=0}^n\|S(t_k)x-B_kx\|^2\right)^{1/2}\le C \|x\|_{\epsilon-\frac{1}{\rho}},~n\Delta t =T,\quad \epsilon >0,
$$
where we also used the fact that $B_0=S(t_0)=I$.
Furthermore, since $\|S(t_k)-B_k\|\le 2$ by Corollary \ref{cor:smooth}, it follows from \eqref{eq:nonsmooth} that
$$
\|S(t_k)x - B_kx \| \leq C \Delta t^{\frac12 -\epsilon} t_k^{\epsilon-\frac12}\|x\|, \quad k\ge 1,
$$
and thus, for some $C=C(\epsilon ,T,\rho)$,
$$
\left(\Delta t\sum_{k=0}^n\|S(t_k)x-B_kx\|^2\right)^{1/2}\le C \Delta t^{\frac12 -\epsilon} \|x\|.
$$
Interpolation finishes the proof.
\end{proof}

%{\color{red} The main result of this section is a strong error estimate for the pure stochastic perturbation, that is, when $u_0=0$. The case $u_0\neq 0$ adds a purely deterministic error term which is then considered in Section \ref{sec:fds} for fully discrete schemes.}

\subsection{Error estimate for the stochastic equation.}We can now state and proof the main result of this section.
\begin{theorem}\label{theo:order}
Let $A$ and $Q$ satisfy \eqref{eq:traceA}--\eqref{eq:Q} and let $b$ satisfy Assumptions \ref{hyp:b} and \ref{hyp:b-anal}. Suppose further that $\E\|u_0\|^2<\infty$. For $T>0$, let $\{u(t)\}_{t\in [0,T]}$ be the unique weak solution of \eqref{eq:stovolterra} and
let $u_n$ be the solution of the scheme \eqref{eq:time_scheme} with $T=n\Delta t$. Then for any $\gamma < (1 - \rho (\alpha - \kappa))/2$, there is $C=C(\rho,\E\|u_0\|^2)>0$ and $K=K(T,\alpha,\gamma,\kappa,\rho)>0$ such that
\begin{equation}\label{eq:timeorder}
%\E \, \| u(T) - u_n \|^2 \leq  \Delta t^{\gamma/2 - \beta/2} \| u_0 \|_{\gamma-1},
(\E \, \| u(T) - u_n \|^2)^{1/2} \leq CT^{-1}\Delta t + K\Delta t^{\gamma}, \quad  t_n=n\Delta t=T.
\end{equation}
\end{theorem}

%\begin{remark}\label{rem:constant}
%Let us note that the dependency in time $T$ of the constant $C$ in (\ref{eq:timeorder}) is due only to the constant $\|b\|_{L^1(0,T)}$. Hence, the dependency in time of the constant $C$
%can be easily removed if the kernel belongs to $L^1(\R_+)$.
%\end{remark}
%\subsection{Proof of Theorem \ref{theo:order}}
\begin{proof}
If
$e_n = u(T)-u_n=u(t_n) - u_n$, then \eqref{eq:stoconv} and \eqref{eq:distvar} yields
$$
e_n =(S(t_n) - B_n) u_0+  \sum_{k=1}^{n} \left [ \int_{t_{k-1}}^{t_{k}} (S(t_n - s) - B_{n-k+1}) \, \dd W^{_Q}(s) \right ].
$$
Taking the expectation of the square of the $H$-norm of $e_n$ leads to, by independence  and  It\^o's isometry:
\begin{equation} \label{eq:time-error}
\E \| e_n \|^2  \le 2(  a + b),
\end{equation}
\noindent where $a$ denotes the deterministic part of the error:
\begin{equation}\label{eq:A}
a = \E\| (S(t_n) - B_n) u_0\|^2,
\end{equation}
\noindent and $b$ the stochastic part:
$$
b  =   \sum_{i=1}^{+\infty }\sum_{k=1}^{n} \int_{t_{k-1}}^{t_{k}} \| (S(t_n - s) - B_{n-k+1})Q^{1/2} e_i \|^2 \,\dd s.
$$
Thanks to (\ref{eq:nonsmooth}), $a$ can be bounded as
\begin{equation}\label{eq:error-a}
a \leq \frac{C}{t_n^2}\Delta t^2  \E\|u_0\|^2, \quad n \geq 1.
\end{equation}
We use Corollary \ref{cor:tee} and Lemma \ref{lem:hol} to bound $b$ as
\begin{align*}
b  & \leq  2\sum_{i=1}^{\infty }\sum_{k=1}^{n} \int_{t_{k-1}}^{t_{k}} \| (S(t_n - s) - S(t_n-t_{k-1}) ) Q^{1/2} e_i \|^2 \, \dd s \\
  &\quad +  2 \sum_{i=1}^{\infty }\sum_{k=1}^{n} \int_{t_{k-1}}^{t_{k}} \| (S(t_n - t_{k-1}) - B_{n-k+1})Q^{1/2} e_i \|^2 \, \dd s\nonumber  \\
& \le C\Delta t^{2\gamma}\sum_{i=1}^{\infty}\|Q^{1/2}e_i\|_{s-\frac{1}{\rho}}=C\Delta t^{2\gamma}\|A^{(s-\frac{1}{\rho})/2}Q^{1/2}\|^2_{\mathcal{L}_2(H)}\\
&\le C\Delta t^{2\gamma}\Tr(A^{s-\frac{1}{\rho}-\kappa})\|A^{\kappa}Q\|^2_{\mathcal{B}(H)}.
\end{align*}
Finally, we set $-\alpha=s-\frac{1}{\rho}-\kappa$ and conclude that $\gamma<\frac{\rho s}{2}=(1 - \rho (\alpha - \kappa))/2$.
\end{proof}
\begin{remark}\label{rem:timeaq}
In particular, if $Q=I$ then $d=1$, $\kappa=0$  and $\alpha>\frac12$ whence $\gamma< 1/2-\frac{\rho}{4}$. Also note, that it is clear from the proof that instead of (\ref{eq:traceA})--(\ref{eq:Q}) we could assume that $\|A^{(s-\frac{1}{\rho})/2}Q^{1/2}\|_{\mathcal{L}_2(H)}<\infty$ and obtain $\gamma<\frac{\rho s}{2}$. Then, for trace class noise; that is, when $\Tr(Q)<\infty$ we can take
$s=\frac{1}{\rho}$ and hence $\gamma<1/2$. Remarkably, this is the same rate as for the heat equation \cite{yan} independently of the value of $\rho$.
\end{remark}

\section{The fully discrete scheme}\label{sec:fds}
In this section we derive strong error estimates for a fully discrete scheme for \eqref{eq:stovolterra}. Both Assumptions 1 and 2 on $b$ are needed but in return we get optimal error bounds with no initial regularity. As the fully discrete scheme, similarly to the time semidiscretization  \eqref{eq:time_scheme}, we consider the recurrence
\begin{equation} \label{eq:full_scheme}
u_{n,h} - u_{n-1,h} + \Delta t \left ( \sum_{k=1}^{n} \omega_{n-k}\,  A_h u_{k,h} \right ) = P_h(W^{_Q}(t_n) - W^{_Q}(t_{n-1})), \quad n\geq 1,
\end{equation}
with $u_{0,h}=P_hu_0$.
Again, the solution is given by the discrete variation of constants formula
\begin{equation}
u_{n,h} = B_{n,h}P_hu_0 + \sum_{k=0}^{n-1} B_{n-k,h} P_h \Delta W^Q_{k+1},
\end{equation}
\noindent where $\Delta W_{k+1} = W(t_{k+1}) - W(t_k)$ and $\{B_{k,h}\}_{k\geq 0}$ is a family of linear bounded operators with $B_{0,h}=I$.
%Define the global quadrature error by
%$$
%\Theta_T(f):=\Delta t\sum_{n=1}^n\left\|\sum_{k=1}^{n} \omega_{n-k}f(k\Delta t)-\int_0^Tb(T-s)f(s)\,\dd s\right\|,~n\Delta t=T.
%$$
\begin{theorem}\label{theo:full}
Let $A$ and $Q$ satisfy (\ref{eq:traceA})--(\ref{eq:Q}) and let $b$ be satisfy Assumptions \ref{hyp:b} and \ref{hyp:b-anal}.
Suppose further that $\E\|u_0\|^2<\infty$. For $T>0$, let $\{u(t)\}_{t\in [0,T]}$ be the unique weak solution of \eqref{eq:stovolterra} and
let $u_{n,h}$ be the solution of the scheme \eqref{eq:full_scheme} with $T=n\Delta t$. If \eqref{eq:Ritz} holds, then there is $C=C(\rho,\E\|u_0\|^2)>0$ and $K=K(T,\alpha,\gamma, \kappa, \rho)>0$ such that
\begin{equation}\label{eq:fullerror}
\left(\mathbb{E}\|u(T)-u_{n,h}\|^2\right)^{1/2}\le C(\Delta t T^{-1}+h^2T^{-\rho})+K(\Delta t^{\gamma}+h^\nu),~n\Delta t=T,
\end{equation}
where $\gamma<(1-\rho(\alpha-\kappa))/2$ and $\nu\le\frac{1}{\rho}-\alpha+\kappa$.
\end{theorem}
\begin{proof}
We decompose the error as
\begin{align*}
u(T)-u_{n,h}&=S(T)u_0-B_{n,h}P_hu_0\\
&+ \int_0^TS(T-s)\,\dd W^Q(s)
-\int_0^TS_h(T-s)P_h\,\dd W^Q(s)\\
&+\int_0^TS_h(T-s)P_h\,\dd W^Q(s)-\sum_{k=0}^{n-1} B_{n-k,h} P_h \Delta W^Q_{k+1}\\
&:=e_1+e_2+e_3.
\end{align*}
First we bound $e_1$ which is the deterministic error. Under Assumptions \ref{hyp:b} and \ref{hyp:b-anal} we have that
$$
(\E\|e_1\|^2)^{1/2}\le C(\Delta t T^{-1}+h^2T^{-\beta-1})(\E\|u_0\|^2)^{1/2}
$$
by \cite[Theorems 2.1 and 3.2]{Lubich_et_al1996}.
Next, $e_2$ has already been bounded in \eqref{eq:e2} as
\begin{equation}\label{eq:ee2}
\E\|e_2\|^2\le Ch^{2\nu}\|A^{(\nu-\frac{1}{\rho})/2}Q^{1/2}\|_{{\mathcal L}_2(H)}^2\le  Ch^{2\nu}\Tr(A^{\nu-\frac{1}{\rho}-\kappa})\|A^{\kappa}Q\|.
\end{equation}
Finally, the proof of Theorem \ref{theo:order} shows that,
\begin{equation}\label{eq:ee3}
\E\|e_3\|^2\le C \Delta t^{2\gamma}  \| A_h^{(s-\frac{1}{\rho})/2} (P_hQP_h)^{1/2}\|^2_{{\mathcal L}_2(H)}.
\end{equation}
Set $-r=(s-\frac{1}{\rho})/2$ and note that since $0<s\le \frac{1}{\rho}$ we have that $0\le r < 1/2$. Then,
\begin{equation*}
\begin{aligned}
\|A_h^{-r}(P_hQP_h)^{1/2}\|_{{\mathcal L}_2(H)}^2
&=\Tr(P_hA_h^{-r}P_hQP_hA_h^{-r}P_h)=\|A_h^{-r}P_hQ^{1/2}\|_{{\mathcal L}_2(H)}^2\\
&\le \|A_h^{-r}P_hA^r\|^2_{\mathcal{B}(H)}\|A^{-r}Q^{1/2}\|^2_{{\mathcal L}_2(H)}.
\end{aligned}
\end{equation*}
Thanks to (\ref{eq:Ah}) with $\delta =r \in [0,1/2)$, it follows that $\|A_h^{-r}P_hA^r\|_{\mathcal{B}(H)}\le 1$. Hence,
\begin{equation*}
\E\|e_3\|^2\le C\Delta t^{2\gamma}\|A^{(s-\frac{1}{\rho})/2}Q^{1/2}\|^2_{\mathcal{L}_2(H)}
\le C\Delta t^{2\gamma}\Tr(A^{s-\frac{1}{\rho}-\kappa})\|A^{\kappa}Q\|^2_{\mathcal{B}(H)},
\end{equation*}
and the proof is complete.
\end{proof}
%\begin{remark}
%As we saw in the above proof, the quadrature error appears in the pure deterministic error and has nothing to do with the stochastic perturbation. If one assumes more on the kernel as in \cite{eggermont1992}; that is, $b$ is globally $L_1$, and that $\hat{b}$ can be extended   in an analytic way to a sector beyond the right halfplane with some appropriate bounds, then
%$\Theta_T(Au)\le C\Delta t \int_0^{\infty}\|A\dot{u}\|\le C\|A u_0\|_{2(1+\epsilon)}$. Hence the estimates in the above theorem are still valid without the quadrature error term.
%\end{remark}
%For the Riesz kernel optimal error estimates for the deterministic problem are readily available and hence we have the following result.
%\begin{corollary}
%Let $b(t)=t^{\beta-1}/\Gamma(\beta)$ and assume that $A$ and $Q$ satisfy (\ref{eq:traceA})--(\ref{eq:Q}). If $\E\|u_0\|^2<\infty$, then there is $C=C(\beta,u_0)$ and $K=K(T,\alpha,\kappa)$ such that
%$$
%\left(\mathbb{E}\|u(T)-u_{n,h}\|^2\right)^{1/2}\le C(\Delta t T^{-1}+h^2T^{-\beta-1})+K(\Delta t^{\gamma/2}+h^\nu),~n\Delta t=T,
%$$
%where $\gamma<1-\rho(\alpha-\kappa)$ and $\nu\le\min(2,\frac{1}{\rho}-\alpha+\kappa)$.
%\end{corollary}
%\begin{proof}
%The statement follows the same way as Theorem \ref{theo:full} where the deterministic error $e_1$ can be bounded
%as
%$$
%(\E\|e_1\|^2)^{1/2}\le C(\Delta t T^{-1}+h^2T^{-\beta-1})(\E\|u_0\|^2)^{1/2}
%$$
%by \cite{Lubich_et_al1996}.
%\end{proof}
\begin{remark}\label{rem:fullaq}
We would like to highlight two important special cases. Firstly, if $Q=I$ then $d=1$, $\kappa=0$  and $\alpha>\frac12$. Hence $\nu< \frac{1}{\rho}-\frac{1}{2}$ and ${\gamma}<1/2-\frac{\rho}{4}$. As before,  we could assume, that $\|A^{(\nu-\frac{1}{\rho})/2}Q^{1/2}\|_{\mathcal{L}_2(H)}<\infty$ instead of \eqref{eq:traceA} and \eqref{eq:Q} and get a convergence rate of order $\nu$ is space and $\gamma<\frac{\rho \nu}{2}$ in time. In particular, if $\Tr(Q)<\infty$, then we may set $\nu=\frac{1}{\rho}$. Thus, the time order is almost $1/2$, the same as for the heat equation with trace class noise, but the space order is less than $1$, which is the space order for the heat equation, see \cite{yan}.
\end{remark}
\begin{remark}
The pure time-discretization as well as the fully discrete scheme can be studied for smooth initial data under Assumptions \ref{hyp:b} and \ref{hyp:b-anal} on $b$. Using \cite[Theorem 3.1]{McLThomee93} and \cite[Lemma 3.2]{Lubich_et_al1996} one arrives at the deterministic estimate
\begin{multline}\label{eq:detnon}
\|S(T)u_0-B_{n,h}P_hu_0\|\\
\le C(T)(h^2+k)\left(\|u_0\|_2+\int_0^T\|\dot{S}(s)u_0\|_2\,\dd s + \int_0^{T}\|\ddot{S}(s)u_0\|\,\dd s \right).
\end{multline}
If $u_0\in \mathcal{D}(A)$, then $u(t)=S(t)u_0$ is a strong solution of \eqref{eq:detvolterra}, see \cite[Proposition 1.2]{pruss}; that is, $u(t)=S(t)u_0$ satisfies \eqref{eq:detvolterra} with $f\equiv 0$ for all $t>0$. Then
$$
\ddot{S}(t)u_0+\int_0^t b(t-s)A\dot{S}(s)u_0\,\dd s+b(t)Au_0=0,~t>0,
$$
and thus
$$
\int_0^{T}\|\ddot{S}(s)u_0\|\,\dd s\le C(T) \left(\int_0^T\|\dot{S}(s)u_0\|_2\,\dd s+\|u_0\|_2\right).
$$
Therefore, using stability, interpolation and \eqref{eq:sss}, it follows that
$$
\|S(T)u_0-B_{n,h}P_hu_0\|\le C(T,\epsilon)(h^{s}+k^{s/2})\|u_0\|_{s(1+\epsilon)},~0\le s\le 2.
$$
The latter estimate can be used to replace the first term in the bound \eqref{eq:fullerror} in case $\E\|u_0\|_{s(1+\epsilon)}^2<\infty$, $0\le s\le 2$.
The estimates for the pure time-discretization are analogous using \cite[Theorem 3.1]{Lubich_et_al1996} which states \eqref{eq:detnon} with $h=0$ and $B_{n,0}=B_n$.
\end{remark}

\end{document}